\documentclass [11pt]{article}
\usepackage{amssymb,amsmath,comment,amsthm}

\def\N{\mathbb{N}}

\def\F{\mathbb{F}}

\newtheorem{theorem}{Theorem}[section]%[chapter]
\newtheorem{proposition}[theorem]{Proposition}

\newtheorem{lemma}[theorem]{Lemma}
\newtheorem{definition}[theorem]{Definition}
\newtheorem{remark}[theorem]{Remark}

\newtheorem{notation}[theorem]{Notation}
\newtheorem{exemple}[theorem]{Example}

\begin{document}
\title{On splitting bi-unitary perfect polynomials over~$\F_{p^2}$}
 \author{Luis H. Gallardo, Olivier Rahavandrainy\\
Univ Brest, UMR CNRS 6205\\
Laboratoire de Math\'ematiques de Bretagne Atlantique}
\maketitle
\\
{\bf{Abstract}}
We give all splitting bi-unitary perfect polynomials
over the field $\F_4$ and some splitting ones over $\F_{p^2}$, if $p$ is an odd prime.

{\section{Introduction}}
Let $p$ be a prime number and let $\F_q$ be a finite field of
characteristic $p$ with $q$ elements.
Let $S \in \F_q[x]$ be a nonzero polynomial.
A divisor $D$ of $S$ is called unitary if $\gcd(D,S/D)=1$. We designate by $\gcd_u(S,T)$ the greatest
common unitary divisor of $S$ and $T$.
A divisor $D$ of $S$ is called bi-unitary if $\gcd_u(D,S/D)=1$.
We denote by $\sigma(S)$ (resp. $\sigma^*(S)$, $\sigma^{**}(S)$) the sum of all divisors
(resp. unitary divisors, bi-unitary divisors)
of $S$.
The functions $\sigma$, $\sigma^*$ and $\sigma^{**}$ are all multiplicative.
We say that $S$ is \emph{perfect} (resp. \emph{unitary perfect},
\emph{bi-unitary perfect}) if $\sigma(S) = S$
(resp. $\sigma^*(S)=S$, $\sigma^{**}(S)=S$).
As usual, $\omega(S)$ designates the number of distinct irreducible
factors of $S$.

We denote by $\N$ the set of non-negative
integers and by $\N\sp{*}$ the set of positive integers.

Several studies are done about perfect, unitary and bi-unitary perfect polynomials (see \cite{BeardU}, \cite{Beard-bup}, \cite{Canaday}, \cite{Gall-Rahav-F4},
\cite{Gall-Rahav9},  \cite{Gall-Rahav6}, \cite{Gall-Rahav2} and references therein).

In this paper, we are interested in splitting polynomials over $\F_{p^2}$ which are bi-unitary perfect (b.u.p.).  We treat the case where $p=2$ in Section \ref{casF4} and the general case in Section \ref{casFpdeux}.

\section{Preliminaries}\label{preliminairebup}
Some of  the following results are obvious or (well) known, so we omit
their proofs. See also \cite{Gall-Rahav-bup-Mersenne}.

\begin{lemma} \label{gcdunitary}
Let $T$ be an irreducible polynomial over $\F_{p^2}$ and $k,l\in \N^*$.
Then,  $\gcd_u(T^k,T^l) = 1 \ ($resp. $T^k)$ if $k \not= l \ ($resp. $k=l)$.\\
In particular, $\gcd_u(T^k,T^{2n-k}) = 1$ for $k \not=n$,  $\gcd_u(T^k,T^{2n+1-k}) = 1$ for
any $0\leq k \leq 2n+1$.
\end{lemma}
\begin{lemma} \label{aboutsigmastar2}
Let $T \in \F_{p^2}[x]$ be irreducible. Then\\
i) $\sigma^{**}(T^{2n}) = (1+T^{n+1}) \sigma(T^{n-1}), \
\sigma^{**}(T^{2n+1}) = \sigma(T^{2n+1})$.\\
ii) For any $c \in \N$, $T$ does not divide $\sigma^{**}(T^{c})$.
\end{lemma}

\begin{lemma} \label{multiplicativity}
If $A = A_1A_2$ is b.u.p. over $\F_{p^2}$ and if
$\gcd(A_1,A_2) =~1$, then $A_1$ is b.u.p. if and only if
$A_2$ is b.u.p.
\end{lemma}
\begin{lemma} \label{translation}
If $A$ is b.u.p. over $\F_{p^2}$, then the polynomial $A(x+t)$
is also b.u.p. over~$\F_{p^2}$, for any $t \in \F_{p^2}$.
\end{lemma}
\subsection{Case $\F_4$} \label{prelim-casF4}
\begin{lemma} \label{splitcritere-F4}
i) $\sigma^{**}(x^{2k+1})$ splits over $\F_4$ if and only if \ $2k+1 =N\cdot 2^n-1$ where $N \in \{1,3\}$ and $n \in \N$.\\
ii) $\sigma^{**}(x^{2k})$ splits over $\F_4$ if and only if \ $2k \in \{2,4,6\}$.
\end{lemma}
\begin{proof}
i) holds since $\displaystyle{\sigma^{**}(x^{2k+1}) = \sigma(x^{2k+1}) = \frac{x^{2k+2}-1}{x-1}}$ must split over $\F_4$.\\
ii): One has $\sigma^{**}(x^{2k}) = (1+x^{k+1}) \cdot \sigma(x^{k-1})$, $\sigma(x^{k-1})$
splits if and only if $k = N \cdot 2^n$ where $N \in \{1,3\}$ and $n \in \N$. Now, $1+x^{k+1}$ splits if and only if $2 N \cdot 2^n + 2 = 2k+2 = M\cdot 2^m$, where $M \in \{1,3\}$ and $m \in \N$.\\
Therefore, $2^{n+1} + 2$ and $3\cdot 2^{n+1} + 2$ belong to the set $\{2^m, 3\cdot 2^m: m \in \N\}$. So, ($N=1$, $n=0$ and $m=2$) or ($N=1$, $n=1$ and $m=1$) or ($N=3$, $n=0$ and $m=3$). Thus, $k \in \{1,2,3\}$.
\end{proof}
\begin{remark}
{\emph{We get from Lemma \ref{splitcritere-F4} and for $T \in \{x,x+1,x+\alpha,x+\alpha+1\}$:
\begin{equation} \label{relation1-F4}
\left\{\begin{array}{l}
\sigma^{**}(T^2) = (T+1)^2\\
\sigma^{**}(T^4) = (T+1)^2 (T+\alpha)(T+\alpha+1)\\
\sigma^{**}(T^6) = (T+1)^4 (T+\alpha)(T+\alpha+1)\\
\sigma^{**}(T^{2^n-1}) = (T+1)^{2^n-1}\\
\sigma^{**}(T^{3 \cdot 2^n-1}) = (T+1)^{2^n-1}(T+\alpha)^{2^n}(T+\alpha+1)^{2^n}.
\end{array}
\right.
\end{equation}}}
\end{remark}
We sometimes use the above equalities for a suitable $T$. We recall here the list of all splitting perfect polynomials over $\F_4$.
\begin{lemma} [\cite{Gall-Rahav-F4}, Theorem 3.4] \label{perfectresultF4}
The polynomial $x^h(x+1)^k(x+\alpha)^l(x+\alpha+1)^t$ is perfect over $\F_4$ if and only if one of the following conditions is satisfied:
i) $h=k=2^n-1,\ l=t=2^m-1$ for some $n,m \in \N$,\\
ii) $h=k=l=t=N \cdot 2^n-1$ for some $n\in \N$ and $N \in \{1,3\}$,\\
iii) $h=l=3 \cdot 2^r-1, \ k=t=2\cdot 2^r-1$ for some $r \in \N$.
\end{lemma}
\subsection{General case} \label{prelim-casgene}
Here, $p$ is an odd prime number. We set $\Omega = \Omega_1 \cup \Omega_2 \cup \Omega_3 \cup \Omega_4$ where $$\begin{array}{l}
\Omega_1:=\{N: \text{ $N$ and $2N+2$ both divide $p^2-1$}\}\\
\Omega_2:=\{p N: \text{$N$ and $2p N+2$ both divide $p^2-1$}\}\\
\Omega_3:=\{N: N | (p^2-1),\ 2N+2 = M\cdot p, \ M | (p^2-1)\}\\
\Omega_4:=\{N: N | (p^2-1),\ 2N+2 = M\cdot p^2, \ M | (p^2-1)\}.
\end{array}$$
\begin{lemma} \label{splitcritere-gene}
i) $\sigma^{**}(x^{2k+1})$ splits over $\F_{p^2}$ if and only if \ $2k+1 =N\cdot p^n-1$ where $N$
divides $p^2-1$ and $n \in \N$.\\
ii) $\sigma^{**}(x^{2k})$ splits over $\F_{p^2}$ if and only if \ $k \in \Omega$.
\end{lemma}
\begin{proof} i) holds since $\displaystyle{\sigma^{**}(x^{2k+1}) = \sigma(x^{2k+1}) = \frac{x^{2k+2}-1}{x-1}}$ must split over $\F_{p^2}$.\\
ii): One has $\sigma^{**}(x^{2k}) = (1+x^{k+1}) \cdot \sigma(x^{k-1})$, $\sigma(x^{k-1})$
splits if and only if $k = N p^n$ where $N$ divides $p^2-1$ and $n \in \N$. Now, $1+x^{k+1}$ splits if and only if $2Np^n + 2 = 2k+2 = Mp^m$, where $M$ divides $p^2-1$ and $m \in \N$.\\
- If $n,m \geq 1$, then $p$ divides $2$, which is impossible.\\
- If $n \geq 1$ and $m=0$, then $2Np^n + 2 = M < p^2$ and thus $n=1$ and $k = N \in \Omega_2$.\\
- If $n = 0$, then  $p^3 > 2p^2+2 > 2N + 2 = Mp^m$ so that $m \in \{0,1,2\}$ and $k \in \Omega_1 \cup \Omega_3 \cup \Omega_4$.
\end{proof}
\begin{lemma} \label{lesOmega}
One has $\Omega_2 = \{p\}$, $\Omega_3 = \{p-1\}$ and $\Omega_4 = \{p^2-1\}$.
\end{lemma}
\begin{proof}
We clearly see that $p \in \Omega_2,\ p-1 \in \Omega_3$  and $p^2-1 \in \Omega_4$. Now, if $pN \in \Omega_2$, we prove that $N = 1$. Since $pN + 1$ divides $(p-1)(p+1)$, we write $(p-1)(p+1) = (pN+1) \cdot \omega$. Thus, $\omega \equiv -1 \mod p$. Set $w = \lambda p-1$. One has $(p-1)(p+1) = (pN+1)( \lambda p-1)$ with $\lambda, N \geq 1$. It follows that $\lambda = N = 1$ and $p=pN \in \Omega_2$.\\
If $N \in \Omega_3$, then $M$ must be even, $\displaystyle{N+1 =\frac{M}{2} p}$ with $N$ and $M$ divide $p^2-1$. Thus, $\displaystyle{(p-1)(p+1) = N \omega = (\frac{M}{2} p -1) \omega}$ and $\omega \equiv 1 \mod p$. Set $w = \lambda p+1$. One has $\displaystyle{(p-1)(p+1) = (\frac{M}{2} p -1)(\lambda p+1)}$ with $\displaystyle{\lambda, \frac{M}{2} \geq 1}$. It follows that $\displaystyle{\lambda = \frac{M}{2} = 1}$ and $p-1 = N \in \Omega_3$.\\
Finally, if $N \in \Omega_4$, then $2p^2 \geq 2N+2 = Mp^2$, $M \leq 2$ and thus $M=2$, $p^2-1 = N \in \Omega_4$.
\end{proof}
\begin{proposition} [\cite{Beard-bup}] \label{bupoverFp}
Let $\displaystyle{A = (x^p-x)^{r}}$ be b.u.p. over $\F_p$. 
Then, $r$ satisfies one of the following conditions:\\
i) $r = N p^n - 1 \equiv 1 \mod 2$, $N$ divides $p-1$ and $n \in \N$,\\
ii) $r = 2(p-1)$,\\
iii) $r = 2N$, $N$ even and $N(N+1)$ divides $p-1$,\\
iv) $p \equiv 1 \mod 4$,\ $r = 2N$, $N$ odd and $2N(N+1)$ divides $p-1$.
\end{proposition}

Since the product of coprime (bi-unitary) perfect polynomials is (bi-unitary) perfect,
the following definition is useful for splitting polynomials.
\begin{definition}
{\emph{We say that $A$ is trivially (bi-unitary) perfect over $\F_{p^2}$ if it
is (bi-unitary) perfect and if it may be written as a product: $A =
\displaystyle{A_0 \cdots A_r},$ where $r \geq 1$ and each $A_i$ is (bi-unitary) perfect over $\F_{p^2}$ and $\gcd(A_i,A_j) = 1 \text{ if } i \not= j$.}}
\end{definition}
\begin{definition}
{\emph{We say that $A$ is indecomposable b.u.p. (i.b.u.p.) if it is b.u.p. but not trivially b.u.p.}}.
\end{definition}

\begin{notation} \label{rootsnotation}
{\emph{We fix
an algebraic closure $\overline{\F_p}$ of $\F_p$. We put
$$\begin{array}{l}
\text{$q=p^2$, $\F_{q} = \{i+j \alpha: i, j \in \F_p \} =
\F_p[\alpha]$,}
\text{ where $\alpha \in \overline{\F_p}$ is a root of $x^2 - c$} \\
\text{and $c$ is not a square in $\F_p$}.
\end{array}$$}}
{\emph{We denote by\\
$\zeta_2, \ldots, \zeta_N$ the $N$-th roots of
$1$ distinct from $1$, if $N$ divides $q-1$.\\
$\beta_1, \ldots, \beta_{N+1} \in \F_q$ the $(N+1)$-th roots (counted with multiplicity) of
$-1$ if $2N+2$ divides $q-1$.\\}}
{\emph{We get from Lemmas \ref{splitcritere-gene} and \ref{lesOmega} and for $T=x -\gamma$, $\gamma \in \F_{p^2}$:
\begin{equation} \label{relation1-gene}
\left\{\begin{array}{l}
\sigma^{**}(\displaystyle{T^{N \cdot p^n-1}) = (T-1)^{p^n-1} \prod_{k=2}^N (T-\zeta_k)^{p^n}} \ \
\text{if $N$ is even}\\
\displaystyle{\sigma^{**}(T^{2N}) = \prod_{k=1}^{N+1} (T-\beta_k) \cdot \prod_{k=2}^N (T-\zeta_k)} \text{ if $N \in \Omega_1$}\\
\displaystyle{\sigma^{**}(T^{2p}) = \prod_{k=1}^{p+1} (T-\beta_k) \cdot (T-1)^{p-1}}\\
\displaystyle{\sigma^{**}(T^{2(p-1)}) = (T+1)^{p} \cdot \sigma(T^{p-2}) = (T+1)^{p+1} \cdot \prod_{\ell=2}^{p-2} (T-\ell)}\\
\displaystyle{\sigma^{**}(T^{2(q-1)}) = (T+1)^{q+1} \cdot \prod_{\ell=2}^{q-2} (T-\ell)}.
\end{array}
\right.
\end{equation}}}
\end{notation}

We sometimes use the equalities in (\ref{relation1-gene}) for a suitable $T$. We recall some results about splitting perfect polynomials over $\F_{p^2}$.

\begin{lemma} [\cite{Gall-Rahav6}, Theorem 1.1] \label{perfectresult}
Let $N$ be a divisor of $p^2-1$ and let $\displaystyle{A = \prod_{\gamma \in \F_{p^2}} (x-\gamma)^{Np^{n(\gamma)} -1}}$ be perfect over $\F_{p^2}$, where $p$ is odd. Then,\\
i) $A$ is trivially perfect if $N \mid (p-1)$,\\
ii) There exists $n \in \N$ such that $n(\gamma)=n$ for all $\gamma \in \F_{p^2}$, if $N \nmid (p-1)$.
\end{lemma}
\section{Case $\F_4$} \label{casF4}
In this section, we work over the finite field  $\F_4$ of $4$ elements:
$$\text{$\F_4 = \{0,1,\alpha,\alpha+1\}$ where $\alpha^2+\alpha+1 = 0$.}$$
We fix $A = x^a(x+1)^b(x+\alpha)^c(x+\alpha+1)^d \in \F_4[x]$, where $a,b,c,d \in \N$.\\
We set
$$\begin{array}{l}
{\mathcal{T}}:=\{r \in \N^*: \text{$r = 2$ or $r =2^n-1$, for some $n\in \N^*$}\}\\
\Sigma:= \{(x^2+x)^r, (x^2+x+1)^r: r \in {\mathcal{T}}\}.
\end{array}$$
If $a,b,c,d$ are all odd, then $\sigma^{**}(A) = \sigma(A)$. So, $A$ is b.u.p. if and only if it is perfect. We apply Lemma \ref{perfectresultF4}.
If $\omega(A) = 2$ or ($\omega(A) \geq 3$ and $A$ trivially b.u.p.), then we apply  Proposition \ref{casebuptrivial}. For the remaining case, we apply Theorem \ref{casebup}. Finally, by Lemma \ref{translation}, we get all splitting b.u.p. over $\F_4$.

\begin{proposition} \label{casebuptrivial}
Let $A \in \F_4[x]$ be b.u.p.. Then\\
i) $A  \in \Sigma$ if $\omega(A) = 2$.\\
ii) $A=A_1 A_2$ where $A_1, A_2 \in \Sigma$ if $A$ is trivially b.u.p. and $\omega(A) \geq 3$.
\end{proposition}

\begin{theorem} \label{casebup}
Let $A = x^a(x+1)^b(x+\alpha)^c(x+\alpha+1)^d \in \F_4[x]$, where $a,b,c,d \in \N$ are not all odd. Then
$A$ is i.b.u.p. if and only if $a, b,c,d$ are given in Table $(\ref{expovalues})$.
\end{theorem}
\begin{equation} \label{expovalues}
\begin{array}{|l|c|c|c|c|c|c|c|c|c|c|c|c|c|}
\hline
a&4&4&6&6&4&4&4&4&4&4&6&6\\
\hline
b&4&4&6&6&3&5&3&5&4&4&6&6\\
\hline
c&3&5&3&5&4&4&3&5&4&6&4&6\\
\hline
d&5&3&5&3&3&5&4&4&4&6&4&6\\
\hline
\end{array}
\end{equation}

\subsection{Proof of Proposition \ref{casebuptrivial}}
\begin{lemma} \label{bup-omega2}
If $\omega(A) = 2$, then $A$ is b.u.p. over $\F_4$ if and only if $A \in \Sigma$.
\end{lemma}
\begin{proof} Sufficiency is obtained by direct computations. For the necessity, without loss of generality, we may assume that $A = x^a(x+\delta)^b$. One has $\sigma^{**}(x^{a}) \cdot \sigma^{**}((x+\delta)^{b})=x^a(x+\delta)^b$. If $a$ is odd, then $\sigma^{**}(x^{a}) = \sigma(x^{a})$ must split over $\F_4$. So, $a = 2^n-1$, for some $n\geq 1$. Thus,  $\sigma^{**}(x^{a}) = (x+1)^{2^n-1} = (x+\delta)^b$. So, $b=2^n-1$, $\beta = 1$ and $A = (x^2+x)^{2^n-1} \in \Sigma$. If $a$ is even, then $a=2k$, 
$k \in \{1,2,3\}$ because $\sigma^{**}(x^{2k})$ must split over $\F_4$. If $k>1$, then $\omega(A) \geq \omega(\sigma^{**}(x^{2k})) =3$, which is impossible. So, $k=1$ and $\sigma^{**}(x^{2k}) = \sigma^{**}(x^{2}) = (x+1)^2$. Thus, $\delta = 1$, $b=2$ and $A = x^2(x+1)^2 \in \Sigma$.
\end{proof}
\begin{lemma} \label{trivial-bup}
If $A = x^a(x+1)^b(x+\alpha)^c(x+\alpha+1)^d$ is trivially b.u.p., where $\omega(A) \geq 3$, then
$\omega(A) = 4$, $a,b,c,d \in {\mathcal{T}}$, $a=b$ and $c=d$.
\end{lemma}
\begin{proof}
We may write $A = A_1 A_2$ where $A_1$ and $A_2$ are b.u.p., $\omega(A_1), \omega(A_2) \geq 2$. So,
$\omega(A_1) =\omega(A_2) = 2$ and $A_1, A_2 \in \Sigma$ by Lemma \ref{bup-omega2}.
\end{proof}
\subsection{Proof of Theorem \ref{casebup}} \label{bupsection}
Sufficiency is obtained by direct computations. For the necessity,
put $A_1= x^a(x+1)^b$ and $A_2=(x+\alpha)^c(x+\alpha+1)^d$ where at least, one of these exponents is even. We also assume that $A$ is i.b.u.p.. So,
$$(a,b), (a,c), (a,d), (b,c), (b,d), (c,d) \not\in \{(2,2), (2^n-1,2^n-1): n \in \N^*\}.$$
We give upper bounds of $a,b,c,d$ without considering too many cases. We  finish by {\tt{Maple}} Computations.

\begin{lemma} \label{reduction1}
i) If $a$ is even and $b$ odd, then $a \in \{2,4,6\}$ and $b \leq 11$.\\
ii)  If $a$ is odd and $b$ even, then $a \leq 11$ and $b \in \{2,4,6\}$.\\
iii) If $c$ is even and $d$ odd, then $c \in \{2,4,6\}$ and $d \leq 11$.\\
iv) If $c$ is odd and $d$ even, then  $c \leq 11$ and $d \in \{2,4,6\}$.
\end{lemma}
\begin{proof}
It suffices to prove i). If $a$ is even and $b$ odd, then $a \leq 6$. Put $b =N\cdot 2^n-1$, with $N \in \{1,3\}$.
By considering the exponents of $x$ in $A$ and in $\sigma^{**}(A)$,
one has $2^n-1 \leq a \leq 6$. So, $n \leq 2$ and $b \leq 3\cdot 2^2-1 = 11$.
\end{proof}
\begin{lemma} \label{reduction2}
i) If $a, b$ are odd and $c$ even, then $a \leq 11, \ b \leq 23$, $c \in \{2,4,6\}$ and $d \leq 11$.\\
ii) If $c, d$ are odd and $a$ even, then $a \in \{2,4,6\}$,  $c \leq 11, \ d \leq 23$ and $b \leq 11$.
\end{lemma}
\begin{proof} It suffices to prove i). First, at least $c$ or $d$ is even. Suppose that $c$ is even.
Put $a=N\cdot 2^n-1$ and $b =M\cdot 2^m-1$, with $N,M \in \{1,3\}$.
If $N=M=1$, then $x^a(x+1)^b$ is perfect and thus $n=m$ and it is b.u.p.\\
Now, if $N=3$, then $2^n \leq c \leq 6$. So, $n \leq 2$ and $a \leq 11$. Moreover, $2^m-1 \leq a\leq 11$, so $m \leq 3$ and $b \leq 23$.\\
If $d$ is even, then $d\leq 6$. If $d$ is odd, then $d \leq 11$ by Lemma \ref{reduction1}-iii).
\end{proof}
Without loss of generality, it suffices to consider the following six cases:
$$\begin{array}{l}
\text{I: only $a$ even, II: $a$ and $b$ even, III: $a$ and $c$ even, }\\
\text{IV: $a$ and $d$ even, \ V: $a$, $b$ and $c$ even, \ VI: $a, b,c,d$ are all even}.
\end{array}$$
\begin{proposition} \label{allcriteria}
i)  If only $a$ is even, then
$a \leq 6, b \leq 11$ and $c,d \leq 23$.\\
ii) If only $a$ and $b$ are even, then
$a,b \leq 6$, $(a,b) \not= (2,2)$ and $c,d \leq 11$.\\
iii) If only $a$ and $c$ are even, then
$a,c \leq 6$, $(a,c) \not= (2,2)$ and $b,d \leq 11$.\\
iv)  If only $a$ and $d$ are even, then
$a,d \leq 6$, $(a,d) \not= (2,2)$ and $b,c \leq 11$.\\
v) If only $d$ is odd, then
$a,b,c \leq 6$, $(a,b), (a,c), (b,c) \not= (2,2)$ and $d \leq 11$.\\
vi)  If $a,b,c,d$ are all even, then $a,b,c,d \leq 6$ and \\
$(a,b), (a,c), (a,d), (b,c), (b,d), (c,d) \not= (2,2)$.
\end{proposition}
\begin{proof} i): One has $a \in \{2,4,6\}$ and $b = N \cdot 2^n-1, c=M \cdot 2^m-1 , d =R \cdot 2^r-1 $, $2^n-1 \leq a \leq 6$, so $n\leq 2$ and $b \leq 3 \cdot 2^2-1 = 11$.\\
If $M=R=1$, then $x^{2^n-1} \| \sigma^{**}(A) = A$. So, we get the contradiction: $2^n-1 = a$ is even. Thus, we may suppose that $M=3$.
Hence, $x^{2^m}$ divides $\sigma^{**}(A) = A$. One has $2^m \leq a\leq 6$, $m \leq 2$, $M \leq 11$, $2^r-1 \leq c\leq 11$, $r \leq 3$ and $d \leq 23$.\\
The proofs of ii), iii), iv) and v) are similar. Note that if $a=b=2$ for example in ii), then $A_1$ is b.u.p. and thus $A_2$ is also b.u.p., which contradicts the fact that $A$ is indecomposable.
\end{proof}

\subsection{Maple Computations} \label{compute}
We search
all $S=x^a(x+1)^b (x+\alpha)^c(x+\alpha+1)^d$ such that $a,b,c,d$ are not all odd, $\omega(S) \geq 3$ and
$\sigma^{**}(S)= S$, by means of Proposition \ref{allcriteria}. We quickly obtain the results stated in Theorem \ref{casebup}.\\
\\
{\bf{The function $\sigma^{**}$ is defined as}} Sigm2star
\begin{verbatim}
> Sigm2star1:=proc(S,a) if a=0 then 1;else if a mod 2 = 0
then n:=a/2:sig1:=sum(S^l,l=0..n):sig2:=sum(S^l,l=0..n-1):
Factor((1+S)*sig1*sig2) mod 2:
else Factor(sum(S^l,l=0..a)) mod 2:fi:fi:end:
> Sigm2star:=proc(S) P:=1:L:=Factors(S) mod 2:k:=nops(L[2]):
for j to k do S1:=L[2][j][1]:h1:=L[2][j][2]:
P:=P*Sigm2star1(S1,h1):od:P:end:
\end{verbatim}
\section{General case} \label{casFpdeux}
We fix an odd prime number $p$ and we put $q=p^2$.
We consider the set $\Omega = \Omega_1 \cup \Omega_2 \cup \Omega_3 \cup \Omega_4$ defined in Section \ref{prelim-casgene}.
Our result is
\begin{theorem} \label{splitbup}
Let $r \in \N^*$ and $A = (x^{q}-x)^{2r} \in \F_{q}[x]$. Then $A$ is b.u.p. over $\F_{q}$ if and only if $r \in \Omega$.
\end{theorem}

\subsection{Proof of Theorem \ref{splitbup}}
We get the necessity from Lemmas \ref{splitcritere-gene} and \ref{lesOmega}, because each $\sigma^{**}((x-\gamma)^{2r})$ must split over $\F_q$. We prove the sufficiency.
\subsubsection{The tuples $\Lambda_N^{\gamma}$, $\Delta_N^{\gamma}$, $\Lambda_{p}^{\gamma}$,
$\Delta_{p}^{\gamma}$, $\Gamma_{p-1}^{\gamma}$ and $\Gamma_{q-1}^{\gamma}$}

According to Notation \ref{rootsnotation}, we consider for $\gamma \in \F_q$,
the following tuples:
$$\begin{array}{l}
\Lambda_N^{\gamma}:= (\gamma - \zeta_2, \ldots,
\gamma - \zeta_N), \text{  if $N \mid (q-1)$},\\
\text{the $(N-1)$-tuples of $\delta \in \F_q$ such that $(x-\gamma)$ divides $\sigma((x-\delta)^{N-1})$, }\\
\\
\Delta_N^{\gamma}:= (\gamma - \beta_1, \ldots,
\gamma - \beta_{N+1}) \text{ if $(2N+2) \mid (q-1)$},\\
\text{the $(N+1)$-tuples of $\delta \in \F_q$ such that $(x-\gamma)$ divides $1+(x-\delta)^{N+1}$, }\\
\\
\Lambda_{p}^{\gamma}:= (\underbrace{\gamma-1, \ldots,\gamma-1}_{(p-1)-times}),\\
\text{the $(p-1)$-tuples of $\delta \in \F_q$ such that $(x-\gamma)$ divides $\sigma((x-\delta)^{p-1})$, }\\
\\
\Delta_{p}^{\gamma}:= (\underbrace{\gamma-\beta_1, \ldots,\gamma-\beta_{p+1}}_{(p+1)-times}),\\
\text{the $(p+1)$-tuples of $\delta \in \F_q$ such that $(x-\gamma)$ divides $1+(x-\delta)^{p+1}$, }\\
\\
\Gamma_{p-1}^{\gamma}:= (\underbrace{\gamma+1, \ldots,\gamma+1}_{p-times}),\ 
\Gamma_{q-1}^{\gamma}:= (\underbrace{\gamma+1, \ldots,\gamma+1}_{q-times}).
\end{array}$$
We denote by $\Pi[j]$ the $j$-th element of a tuple $\Pi$, $1 \leq j \leq {\tt{length}}(\Pi)$.

\begin{exemple}
{\emph{$\Lambda_{p-1}^{\gamma} = (\gamma-2, \gamma-3, \ldots,\gamma-(p-1))$,\\
$\text{$\Lambda_N^{\gamma}[j]=\gamma - \zeta_{j+1}$, $1 \leq j \leq N-1$,}$\
$\Delta_N^{\gamma}[j]=\gamma - \beta_j$, $1 \leq j \leq N+1$,\\
$\Lambda_p^{\gamma}[j]=\gamma - 1$, $1 \leq j \leq p-1$, $\Gamma_{q-1}^{\gamma}[j] = \gamma +1$, $1 \leq j \leq q$,\\
For $N \in \Omega_1$, $x-\gamma$ divides $\sigma^{**}((x-\delta)^{2N})$ if and only if
$\delta = \Lambda_N^{\gamma}[j]$,  for some $j$ or $\delta = \Delta_N^{\gamma}[k]$, for some $k$.}}
\end{exemple}
The following straightforward result is useful.
\begin{lemma} \label{lemma2beard}
A polynomial $Q$ is bi-unitary perfect if and only if for any
irreducible polynomial $P \in \F_q[x]$, and for any positive
integers $m_1, m_2$, we have:$$(P^{m_1} \ || \ Q, \ P^{m_2} \ || \
\sigma^{**}(Q))~\Longrightarrow~(m_1~=~m_2).$$
\end{lemma}
We obtain an immediate consequence which finishes the proof of Theorem \ref{splitbup}.
\begin{proposition} \label{perfectcritere}
Let $m \in \N^*$ and $A = \displaystyle{\prod_{\gamma
\in \F_q} (x-\gamma)^{2m}} \in \F_{q}[x]$, where $m \in \Omega$. Then, $A$ is b.u.p. over $\F_q$.
\end{proposition}
\begin{proof}
For every $\gamma \in \F_q$, we may apply Lemma \ref{lemma2beard} to
the polynomial $x-\gamma$, where $m_1 = 2m \geq 1$. \\
- If $m \in \Omega_1$, then $m=N$ where $N$ and $2N+2$ both divide $q-1$.
From (\ref{relation1-gene}), we get the equality: $$\sigma^{**}(A) = \displaystyle{\prod_{\delta \in
\F_q} \sigma^{**}((x-\delta)^{2N}) = \prod_{k=1}^{N+1} (T-\beta_k) \cdot \prod_{k=2}^N (T-\zeta_k)},$$ we see that $x-\gamma$ divides $\sigma^{**}(A)$ if and only if it divides $\sigma^{**}((x-\delta)^{2m})$ for some $\delta \in \F_q$. Hence, $\gamma$ is of the form $\delta+\zeta_j$ or $\delta+\beta_k$
for some $\zeta_j$ and $\beta_k$, i.e., $\delta = \Lambda_N^{\gamma}[j]$ or $\delta =\Delta_N^{\gamma}[k]$. Hence, the exponent of
$x-\gamma$ in $\sigma^{**}(A)$ is exactly the integer $m_2 =
\displaystyle{{\tt{length}} (\Lambda_N^{\gamma}) + {\tt{length}} (\Delta_N^{\gamma}) = (N-1) + (N+1) = 2m}.$\\
- If $m \in \Omega_2$, then $m=p$. 
The monomial $x-\gamma$ divides $\sigma^{**}(A)$ if and only if there exists $\delta \in \F_q$ such that $\delta = \Lambda_p^{\gamma}[j]$ or $\delta =\Delta_{p}^{\gamma}[k]$, for some $j,k$. We get $$m_2 =
\displaystyle{{\tt{length}}(\Lambda_{p}^{\gamma}) + {\tt{length}} (\Delta_{p}^{\gamma}) = (p-1)+(p+1) = 2m}.$$
- If $m \in \Omega_3$, then $m= p-1$. 
The monomial $x-\gamma$ divides $\sigma^{**}(A)$ if and only if there exists $\delta \in \F_q$ such that $\delta = \Lambda_{p-1}^{\gamma}[j]$ or $\delta =\Gamma_{p-1}^{\gamma}[k]$, for some $j,k$. We get $$m_2 =
\displaystyle{{\tt{length}} (\Lambda_{p-1}^{\gamma}) + {\tt{length}} (\Gamma_{p-1}^{\gamma}) = (p-2)+(p) = 2m}.$$
- If $m \in \Omega_4$, then $m=q-1$. 
The monomial $x-\gamma$ divides $\sigma^{**}(A)$ if and only if there exists $\delta \in \F_q$ such that $\delta = \Lambda_{q-1}^{\gamma}[j]$ or $\delta =\Gamma_{q-1}^{\gamma}[k]$, for some $j,k$. We get $$m_2 =
\displaystyle{{\tt{length}} (\Lambda_{q-1}^{\gamma}) + {\tt{length}} (\Gamma_{q-1}^{\gamma}) = (q-2)+(q) = 2m}.$$
\end{proof}
\def\thebibliography#1{\section*{\titrebibliographie}
\addcontentsline{toc}
{section}{\titrebibliographie}\list{[\arabic{enumi}]}{\settowidth
 \labelwidth{[
#1]}\leftmargin\labelwidth \advance\leftmargin\labelsep
\usecounter{enumi}}
\def\newblock{\hskip .11em plus .33em minus -.07em} \sloppy
\sfcode`\.=1000\relax}
\let\endthebibliography=\endlist

\def\biblio{\def\titrebibliographie{References}\thebibliography}
\let\endbiblio=\endthebibliography

%%%% MACROS DE SEROUL POUR LES REFERENCES %%%%

%%%%%%% bibliographie selon AMS style %%%%%%%%%%%
%%%%%%% inspir de TUGboat 11 (1990), p. 609 %%%%%%%

\newbox\auteurbox
\newbox\titrebox
\newbox\titrelbox
\newbox\editeurbox
\newbox\anneebox
\newbox\anneelbox
\newbox\journalbox
\newbox\volumebox
\newbox\pagesbox
\newbox\diversbox
\newbox\collectionbox
%--------------------------------------------
\def\fabriquebox#1#2{\par\egroup
\setbox#1=\vbox\bgroup \leftskip=0pt \hsize=\maxdimen \noindent#2}
%--------------------------------------------
\def\bibref#1{\bibitem{#1}

%\mbox{}\ignorespaces

\setbox0=\vbox\bgroup}
%--------------------------------------------
\def\auteur{\fabriquebox\auteurbox\styleauteur}
\def\titre{\fabriquebox\titrebox\styletitre}
\def\titrelivre{\fabriquebox\titrelbox\styletitrelivre}
\def\editeur{\fabriquebox\editeurbox\styleediteur}

\def\journal{\fabriquebox\journalbox\stylejournal}

\def\volume{\fabriquebox\volumebox\stylevolume}
\def\collection{\fabriquebox\collectionbox\stylecollection}
%--------------------------------------------
{\catcode`\- =\active\gdef\annee{\fabriquebox\anneebox\catcode`\-
=\active\def -{\hbox{\rm
\string-\string-}}\styleannee\ignorespaces}}
%--------------------------------------------
{\catcode`\-
=\active\gdef\anneelivre{\fabriquebox\anneelbox\catcode`\-=
\active\def-{\hbox{\rm \string-\string-}}\styleanneelivre}}
%--------------------------------------------
{\catcode`\-=\active\gdef\pages{\fabriquebox\pagesbox\catcode`\-
=\active\def -{\hbox{\rm\string-\string-}}\stylepages}}
%--------------------------------------------
{\catcode`\-
=\active\gdef\divers{\fabriquebox\diversbox\catcode`\-=\active
\def-{\hbox{\rm\string-\string-}}\rm}}
%--------------------------------------------
\def\ajoutref#1{\setbox0=\vbox{\unvbox#1\global\setbox1=
\lastbox}\unhbox1 \unskip\unskip\unpenalty}
%--------------------------------------------
\newif\ifpreviousitem
\global\previousitemfalse
\def\separateur{\ifpreviousitem {,\ }\fi}
%--------------------------------------------
\def\voidallboxes
{\setbox0=\box\auteurbox \setbox0=\box\titrebox
\setbox0=\box\titrelbox \setbox0=\box\editeurbox
\setbox0=\box\anneebox \setbox0=\box\anneelbox
\setbox0=\box\journalbox \setbox0=\box\volumebox
\setbox0=\box\pagesbox \setbox0=\box\diversbox
\setbox0=\box\collectionbox \setbox0=\null}
%--------------------------------------------
\def\fabriquelivre
{\ifdim\ht\auteurbox>0pt
\ajoutref\auteurbox\global\previousitemtrue\fi
\ifdim\ht\titrelbox>0pt
\separateur\ajoutref\titrelbox\global\previousitemtrue\fi
\ifdim\ht\collectionbox>0pt
\separateur\ajoutref\collectionbox\global\previousitemtrue\fi
\ifdim\ht\editeurbox>0pt
\separateur\ajoutref\editeurbox\global\previousitemtrue\fi
\ifdim\ht\anneelbox>0pt \separateur \ajoutref\anneelbox
\fi\global\previousitemfalse}
%--------------------------------------------
\def\fabriquearticle
{\ifdim\ht\auteurbox>0pt        \ajoutref\auteurbox
\global\previousitemtrue\fi \ifdim\ht\titrebox>0pt
\separateur\ajoutref\titrebox\global\previousitemtrue\fi
\ifdim\ht\titrelbox>0pt \separateur{\rm in}\
\ajoutref\titrelbox\global \previousitemtrue\fi
\ifdim\ht\journalbox>0pt \separateur
\ajoutref\journalbox\global\previousitemtrue\fi
\ifdim\ht\volumebox>0pt \ \ajoutref\volumebox\fi
\ifdim\ht\anneebox>0pt  \ {\rm(}\ajoutref\anneebox \rm)\fi
\ifdim\ht\pagesbox>0pt
\separateur\ajoutref\pagesbox\fi\global\previousitemfalse}
%--------------------------------------------
\def\fabriquedivers
{\ifdim\ht\auteurbox>0pt
\ajoutref\auteurbox\global\previousitemtrue\fi
\ifdim\ht\diversbox>0pt \separateur\ajoutref\diversbox\fi}
%--------------------------------------------
\def\endbibref
{\egroup \ifdim\ht\journalbox>0pt \fabriquearticle
\else\ifdim\ht\editeurbox>0pt \fabriquelivre
\else\ifdim\ht\diversbox>0pt \fabriquedivers \fi\fi\fi
.\voidallboxes}
%--------------------------------------------

\let\styleauteur=\sc
\let\styletitre=\it
\let\styletitrelivre=\sl
\let\stylejournal=\rm
\let\stylevolume=\bf
\let\styleannee=\rm
\let\stylepages=\rm
\let\stylecollection=\rm
\let\styleediteur=\rm
\let\styleanneelivre=\rm

\begin{biblio}{99}

\begin{bibref}{BeardU}
\auteur{J. T. B. Beard Jr} \titre{Unitary perfect polynomials over
$GF(q)$} \journal{Rend. Accad. Lincei} \volume{62} \pages 417-422
\annee 1977
\end{bibref}

\begin{bibref}{Beard-bup}
\auteur{J. T. B. Beard Jr}  \titre{Bi-Unitary Perfect polynomials over $GF(q)$}
\journal{Annali di Mat. Pura ed Appl.} \volume{149(1)} \pages 61-68 \annee 1987
\end{bibref}

\begin{bibref}{Canaday}
\auteur{E. F. Canaday} \titre{The sum of the divisors of a
polynomial} \journal{Duke Math. J.} \volume{8} \pages 721-737 \annee
1941
\end{bibref}

\begin{bibref}{Gall-Rahav-F4}
\auteur{L. H. Gallardo, O. Rahavandrainy} \titre{On
perfect polynomials over $\F_4$}
\journal{Port. Math. (N.S.)} \volume{62(1)} \pages 109-122 \annee
2005
\end{bibref}

\begin{bibref}{Gall-Rahav6}
\auteur{L. H. Gallardo, O. Rahavandrainy} \titre{On splitting
perfect polynomials over $\F_{p^2}$} \journal{Port. Math. (N.S.) }
\volume{66(3)} \pages 261-273 \annee 2009
\end{bibref}

\begin{bibref}{Gall-Rahav9}
\auteur{L. H. Gallardo, O. Rahavandrainy} \titre{On unitary
splitting perfect polynomials over $\F_{p^2}$} \journal{Math.
Commun.} \volume{15(1)} \pages 159-176 \annee 2010
\end{bibref}

\begin{bibref}{Gall-Rahav2}
\auteur{L. H. Gallardo, O. Rahavandrainy} \titre{On splitting
perfect polynomials over $\F_{p^p}$} \journal{Int. Electron. J.
Algebra} \volume{9} \pages 85-102 \annee 2011
\end{bibref}

\begin{bibref}{Gall-Rahav-bup-Mersenne}
\auteur{L. H. Gallardo, O. Rahavandrainy} \titre{All bi-unitary perfect polynomials over $\F_2$ only divisible by $x$, $x+1$ and by Mersenne primes} \journal{arXiv Math: 2204.13337} \annee 2022
\end{bibref}

\end{biblio}

\end{document}